\documentclass[psamsfont,12pt]{amsart}
\usepackage[usenames]{xcolor}
\usepackage{amsmath,amssymb}
\usepackage{dsfont}\let\mathbb\mathds
\usepackage{xypic}
\usepackage{amsthm,amsfonts,amsmath,amssymb,latexsym,epsfig,mathrsfs,yfonts,marvosym}
\usepackage{graphicx}

\usepackage{hyperref}
\usepackage{dsfont}\let\mathbb\mathds
\usepackage[english]{babel}

\DeclareMathAlphabet\oldmathcal{OMS}        {cmsy}{b}{n}
\SetMathAlphabet    \oldmathcal{normal}{OMS}{cmsy}{m}{n}
\DeclareMathAlphabet\oldmathbcal{OMS}       {cmsy}{b}{n}

\usepackage{hyperref}
\usepackage[english]{babel} 

\newtheorem{theorem}{Theorem}[section]
\newtheorem{lemma}[theorem]{Lemma}
\newtheorem{corollary}[theorem]{Corollary}
\newtheorem{convention}[theorem]{Convention}
\newtheorem{proposition}[theorem]{Proposition}

\theoremstyle{remark}
\newtheorem{rem}[theorem]{Remark}
\theoremstyle{definition}

\DeclareMathOperator{\im}{Im}

\newcommand{\intprod}{\mathbin{\raisebox{\depth}{\scalebox{1}[-1]{$\lnot$}}}}

\newcommand{\nts}[1]{\marginpar{#1}}
\renewcommand{\nts}[1]{}
\newcommand{\R}{\mathbf{R}}

  \def\bfV{\mbox{{\bf V}}} \def\bfS{\mbox{{\bf S}}}    \def\bfH{\mbox{{\bf H}}}      \def\bfF{\mbox{{\bf F}}}
\def\bfZ{\mbox{{\bf Z}}}
\def\mF{\mathcal{F}}   \def\mA{\mathcal{A}}     \def\mL{\mathcal{L}}    \def\mH{\mathcal{H}}           

\def\cala{\mathcal{A}} \def\cals{\mathcal{S}} 
   \def\bC{\mathbb C} \def\bR{\mathbb R} \def\bQ{\mathbb Q}      

\def\R{\mathbb R}\def\C{\mathbb C}

\def\fract#1#2{\raise4pt\hbox{$ #1 \atop #2 $}}

\def\bbc{{\mathbb C}}

\def\bbp{{\mathbb P}}
\def\bbq{{\mathbb Q}}
\def\bbr{{\mathbb R}}

\def\bbz{{\mathbb Z}}

\def\gre{\epsilon}

\def\gro{\omega}

\def\grr{\rho}

\def\grD{\Delta}

\def\grS{\Sigma}

\def\bfw{{\bf w}}

\def\delr{ \frac{\del}{\del r}}

\def\kg{{\mathfrak{g}}}

\usepackage{marvosym}

\def\del{\partial}              

\def\kt{\mathfrak{t}}

\def\vol{d v }

\def\ra{\rightarrow}

\def\Sas{\mbox{Sas}}

\def\Ds{\mathcal{D}}

\def\cala{{\mathcal A}}

\def\cald{{\mathcal D}}

\def\calf{{\mathcal F}}

\def\call{{\mathcal L}}

\def\cals{{\oldmathcal S}}

\def\gre{\epsilon}

\def\gro{\omega}

\def\grr{\rho}

\def\grD{\Delta}

\def\grS{\Sigma}

\def\gr{{\mathfrak r}}

\def\gt{{\mathfrak t}}

\begin{document}

\title[The DH theorem and applications]{An application of the Duistertmaat--Heckman Theorem and its extensions in Sasaki Geometry}
\date{\today}
\author[C. Boyer]{Charles P. Boyer}
\author[H. Huang]{Hongnian Huang}
\author[E. Legendre]{Eveline Legendre}

%
%
 \address{Charles Boyer, Department of Mathematics and Statistics,
 University of New Mexico, Albuquerque, NM 87131.}
 \email{cboyer@math.unm.edu} 
 \address{Hongnian Huang, Department of Mathematics and Statistics,
 University of New Mexico, Albuquerque, NM 87131.}
 \email{hnhuang@unm.edu}
 \address{Eveline Legendre\\ Universit\'e Paul Sabatier\\
 Institut de Math\'ematiques de Toulouse\\ 118 route de Narbonne\\
 31062 Toulouse\\ France}
\email{eveline.legendre@math.univ-toulouse.fr}

 \thanks{The first author was partially supported by grant \#245002 from the Simons Foundation. The third author is partially supported by France ANR project EMARKS No ANR-14-CE25-0010.} 

\begin{abstract} Building on an idea laid out by Martelli--Sparks--Yau in~\cite{MSYvolume}, we use the Duistermaat-Heckman localization formula and an extension of it to give rational and explicit expressions of the volume, the total transversal scalar curvature and the Einstein--Hilbert functional, seen as functionals on the Sasaki cone (Reeb cone). Studying the leading terms we prove they are all proper. Among consequences we get that the Einstein-Hilbert functional attains its minimal value and each Sasaki cone possess at least one Reeb vector field with vanishing transverse Futaki invariant.        
\end{abstract}

\maketitle

\section{Introduction}
 
The general problem motivating our work is: given a polarized Sasaki type manifold $(N^{2n+1},\xi)$, does there exist a compatible constant scalar curvature Sasaki (cscS for short) metric? This is a hard problem and the answer is conjecturally related to some notion of K-stability see~\cite{CollinsSz} and is closely related to the analogous problem in (compact) K\"ahler geometry, see for eg.~\cite{don:scalar,don:csc,RT,stoppa}. Well-known obstructions are the K-stability see~\cite{CollinsSz}, the transversal Futaki invariant~\cite{FutakiOnoWang} and the Einstein--Hilbert functional~\cite{nonUNIQcscS,BHLT_EH}. In this paper we study the latter using the Duistermaat-Heckman localization formula developping on an idea of~\cite{MSYvolume}. More precisely, given a Sasaki manifold $(N,\Ds,J, \xi_o)$ with maximal torus of automorphisms $T\subset \mbox{CR}$, the (reduced) Sasaki cone (or Reeb cone), denoted $\kt^+$, is an open polyhedral cone in $\kt =\mbox{Lie}(T)$ and contains all the $T$--invariant Reeb vector 
fields on $(N,\Ds,J)$. Picking a quasi-regular 
Reeb vector field $\xi_o\in\kt^+$, 
that is, it induces a circle action $S_o^1 \subset T$, we build an orbibundle $\pi: \mL\ra W$ polarizing the quotient symplectic orbifold $(W:=N/S_o^1, \sigma)$ so that there is a biholomorphism between $\mL \backslash W $ and the K\"ahler cone $Y_0:=N\times \bR$ over $N$. The action of $T$ is defined on $L$ and the fixed points set $\bfZ=\sqcup Z$ consists in a disjoint union of symplectic suborbifolds of $(W,\sigma)$. Moreover, the normal bundle of $Z$ in $M$ is $T$--invariant and splits with respect to the action into (real) rank $2$ symplectic bundles $\oplus^{n-n_Z}_{i=0} E_i^Z$ on which $T$ acts with weight $\underline{\kappa}_i^Z \in\kt^*$ (where $\dim Z=2n_Z$). We recall and explain these facts in~\S\ref{SECTorbiResol} and we then we use it to prove that           

\begin{equation}\label{eq:VOL}
\bfV_{\xi} = \mbox{vol}(N,\xi)= \frac{(2\pi)^{n+1}}{n!} \sum_Z \frac{1}{d_Z}\int_Z \frac{1}{\prod_{j=0}^{n-n_Z} \left(c_1(E^Z_j) -\langle \underline{\kappa}_j^Z, \xi\rangle\right)}.
\end{equation}

In~\cite{MSYvolume}, Martelli, Sparks and Yau proved that the volume functional $\bfV$ is a rational function of $\xi\in \kt^+$ when the K\"ahler cone is Gorenstein. This hypothesis is not needed to get \eqref{eq:VOL} and moreover by studying the leading term in \eqref{eq:VOL} we get the following corollary   

 \begin{theorem}\label{theoVol=rat}
  Let $(N, \Ds)$ be a compact contact manifold of Sasaki type of dimension $2n+1$ and $\kt^+$ a compatible Sasaki cone. The volume functional $\bfV: \kt^+ \ra \bR$ is a rational function homogenous of order $-(n+1)$. Moreover, $\bfV$ tends to $+\infty$ when $\xi$ approaches the boundary of $\kt^+$.   
 \end{theorem}

 We also derive a similar formula for the {\it total transversal scalar curvature} $\bfS_\xi$ using the same orbibundle $\call\ra W$ and an extension, see Theorem~\ref{extDHteo}, of the Duistermaat--Heckman localization formula. The formula is  

 \begin{equation}\label{eqSCAL}
\bfS_{\xi} =\frac{2(2\pi)^{n+1}}{(n-1)!} \sum_Z \frac{1}{d_Z}\int_Z \frac{(\iota_Z^*c_1(W) + \sum_{i=1}^{n-n_Z} \langle \underline{\kappa}_i^Z, \xi\rangle  )}{\prod_{j=0}^{n-n_Z} ( c_1(E^Z_j)   -\langle \underline{\kappa}_j^Z, \xi\rangle)} 
\end{equation}

As a direct consequence we get the following Theorem. 

 \begin{theorem}\label{theoSCAL=rat}
  Let $(N, \Ds)$ be a compact contact manifold of Sasaki type of dimension $2n+1$ with a fixed class of compatible CR structures $[J]$. Let $\kt^+$ be a compatible Sasaki cone. The total transversal scalar curvature functional $\bfS: \kt^+ \ra \bR$ is a rational function homogenous of order $-n$. 
 \end{theorem}

 \begin{rem}
  Tian used the Duistermaat-Heckman localization formula to compute the Futaki invariant of some K\"ahler manifolds in~\cite{Tian}. Using the relation found in \cite{nonUNIQcscS,BHLT_EH}, Tian's formula should be related to the derivative of~\eqref{eqSCAL} in the regular case.
 \end{rem}  
 
  \begin{rem}In the toric case the fixed points set consists in a set of isolated points. The formula \eqref{eqSCAL} becomes  
  \begin{equation}\label{eqSCALtoric}
\bfS_{\xi} =\frac{2(2\pi)^{n+1}}{(n-1)!} \sum_p \frac{1}{d_Z} \frac{\sum_{i=1}^{n} \langle \underline{\kappa}_i^p, \xi\rangle }{\prod_{j=0}^{n}\langle \underline{\kappa}_j^p, \xi\rangle)} 
\end{equation}  which coincides with the one found in \cite{nonUNIQcscS} using integration by parts.\end{rem}

 The claims of Theorems \ref{theoVol=rat} and \ref{theoSCAL=rat} may also be obtained by studying the Hilbert series of the polarized K\"ahler cone associated to $(N,\Ds,J,\xi)$, see~\cite{CollinsSz}, as pointed out to us by Tristan Collins.
 
 Now by analysing the weights of the action of $T$ on $\mL$, using the construction and Morse's Lemma we get that the leading term, when $\xi\in \kt^+$ tends to the boundary of $\kt^+$, of the rhs of~\eqref{eqSCAL} is positive. From which we derive the following result. 
 
  \begin{corollary}\label{coroSboundary} When $\xi\in \kt^+$ tends to the boundary of $\kt^+$, $\bfS_{\xi}$ tends to $+\infty$.  
 \end{corollary}
 
 The Einstein--Hilbert functional is defined here to be the following homogenous functional 
\begin{equation}\label{HEdef} 
 \bfH(\xi) = \frac{\bfS^{n+1}_{\xi}}{\bfV_\xi^n}
\end{equation} 
on the (reduced) Sasaki cone $\kt^+$. 

In~\cite{nonUNIQcscS,BHLT_EH} it is shown that this functional detects the Reeb vector field whose transversal Futaki invariant vanishes. That is fixing the isotopy class of $(\Ds,J)$ we have 
$$\{\xi \in \kt \,|\, \exists \mbox{ compatible cscS } \}\subset \mbox{crit}(\bfH).$$
This is our main motivation to study this functional and in particular, its critical points. Non-uniqueness is now known, indeed $\bfH$ is not convex and may possess many critical points~\cite{nonUNIQcscS,BoToJGA,BHLT_EH}. In this paper we solve the question of existence by proving that $\bfH$ is proper on the convex (relatively compact) set of rays of the Sasaki cone.

\begin{corollary}\label{coroH} 
The Einstein--Hilbert functional is an homogenous rational function on $\kt^+$, tending to $+\infty$ when $\xi\in \kt^+$ tends to the boundary of $\kt^+$. In particular, it attains its minimal value along a ray in the interior of the cone $\kt^+$.\end{corollary}
 
   \begin{corollary}\label{coroH2} Any Sasaki cone contains rays of Reeb vector fields with vanishing transversal Futaki invariant. Moreover, any ray of  Reeb vector fields with vanishing transversal Futaki invariant, in particular ray of cscS structures, is isolated in any $2$--dimensional subcone of the Sasaki cone.       
 \end{corollary}
 
 To obtain formula~\eqref{eqSCAL} we relate the total scalar curvature of the metric associated to $\xi$ to the metric associated to $\xi_o$ on $W$ in order to get something independant. The relation we found is quite explicit, see~\eqref{eqSCALT} and we get the following corollary.

\begin{corollary}\label{coroPOS} Within a Sasaki cone there is at most one ray of vanishing transverse scalar curvature. Moreover, if there is one Sasaki structure with non negative transverse scalar curvature then the total transverse scalar curvature is non-negative on the whole Sasaki cone. 
\end{corollary}

The next section gathers facts and convention we used in this paper. We present in Section 3 the localization formula of Duistermaat--Heckman and the extension we need. In \S\ref{SECTorbiResol} we build the orbi-resolution of the cone we will use in the subsequent sections to prove Theorem~\ref{theoVol=rat} in \S\ref{sectAppSasVol} and Theorem~\ref{theoSCAL=rat} in~\S\ref{sectTTSCf}. The final section first gathers the proof of the Corollaries~\ref{coroH}--\ref{coroPOS}. Then we discuss how the positivity (or non-negativity) of $\bfS_\xi$ affects the transverse geometry, namely that in a certain sense the transverse structure is dominated by rational curves. Although the non-negativity of $\bfS_\xi$ is sufficient for this, it is not necessary. \\     

\noindent {\bf Aknowledgment }{\it The third author would like to thank her collegues J.-F. Barraud, P. Gauduchon and R. Leclercq for illuminating discussions on symplectic manifolds and equivariant cohomology. We also thank T.Collins, J.Sparks and S. Sun for their comments on a previous version of this paper and their interest in this work.}

\section{Some facts and conventions in Sasaki geometry}
\subsection{Basics of Sasaki geometry}\label{sectBackground} We give in this section the basic definitions and facts we need for our purposes. We refer to \cite{BG:book} for an extensive study of Sasakian geometry.
\subsubsection{Sasaki cone} A Sasakian manifold of dimension $2n+1$ is a smooth manifold $N$ endowed with the following structures 
\begin{itemize}
 \item a co-oriented rank $2n$ contact distribution $\Ds$;
 \item a CR structure $J\in \mbox{End}(\Ds)$ (i.e $J^2= -\mbox{id}_\Ds$ + integrability condition);  
 \item a nowhere vanishing contact form $\eta \in \Gamma(\Ds^0)$, where $\Ds^0:=\mbox{Annihilator}(\Ds) \subset T^*N$;
 \item a Riemannian metric $g= \eta^2 + \frac{1}{2}d\eta(\cdot,J\cdot)$;
 \item a Reeb vector field $\xi$ lying in the Lie algebra of $CR$--diffeomorphisms $\mathfrak{cr}(\Ds,J)$ and satisfying the conditions $\eta(\xi)=1$ and $\mL_\xi\eta=0$.
\end{itemize}

Any $3$ of these structures determine the remaining ones. For example, given $(\Ds,J,\eta)$ we get $\xi$ by solving the equations involved, then $\Phi\in \Gamma(\mbox{End(TN)})$, defined as $\Phi(\xi)=0$ and $\Phi_{|_{\Ds}}=J$ and finally $$g= \frac{1}{2}d\eta(\cdot,\Phi(\cdot)) + \eta\otimes \eta.$$  Therefore, the space of Sasaki structures sharing the same CR-structure, denoted $\Sas(\Ds,J)$, is in bijection with the cone of Reeb vector fields \begin{equation}\label{eq:SasakiCone}
\mathfrak{cr}^+(\Ds,J)= \{X \in \mathfrak{cr}(\Ds,J)\,|\, \eta (X)>0\}\end{equation}
where $\eta$ is any fixed, nowhere vanishing, section of $\Ds^0$. The map $\mathfrak{cr}^+(\Ds,J)\ra \Sas(\Ds,J)$ is given by \begin{equation}\label{defnSasakiCONE}
\xi' \mapsto \left(\frac{\eta}{\eta(\xi')},\Ds,J\right). \end{equation}                                                                       

To avoid confusion, we sometimes refer to a Sasakian structure redundantly by specifying all the structures involved like $(N,\Ds,J,\eta,g,\xi)$. 
From~\cite{BG:book} we know that $\mathfrak{cr}^+(\Ds,J)$ is an open convex cone in $\mathfrak{cr}(\Ds,J)$, invariant under the adjoint action of $\mathfrak{CR}(\Ds,J)$. Moreover, the following result will be useful for our study. 

\begin{theorem}\cite{BGS} Let $N$ be a compact manifold of dimension $2n + 1$ with a
CR-structure $(\Ds, J)$ of Sasaki type. Then the Lie algebra $\mathfrak{cr}(\Ds,J)$ decomposes as
$\mathfrak{cr}(\Ds,J)=\underline{\kt} +\mathfrak{p}$, where $\underline{\kt}$ is the Lie algebra of a maximal torus $\underline{T}$ of dimension $k$ with $1 \leq k \leq n + 1$ and $\mathfrak{p}$ is a completely reducible $\underline{\kt}$--module. Furthermore, every
$X \in \mathfrak{cr}^+(\Ds,J)$ is conjugate to a positive element in the Lie algebra $\underline{\kt}$. 
\end{theorem}

The (reduced) {\it Sasaki cone} is the set $\underline{\kt}^+=\underline{\kt}\cap  \mathfrak{cr}^+(\Ds,J)$. When seeking extremal or csc Sasaki metrics one deforms the contact structure by $\eta\mapsto \eta +d^c\varphi$ where the function $\varphi$ is invariant under the maximal torus $\underline{T}$. Thus, the Sasaki cone is associated with an isotopy class of contact structures of Sasaki type that is invariant under $\underline{T}$. This fact has given rise to several equivalent definitions of $\gt^+$, see \cite{HeSu12b,CollinsSz,BovC16}, and is also called the {\it Reeb cone}.

\subsubsection{K\"ahler cones} \label{sectCONEbackground}

Recall, see for eg.~\cite{L:contactToric}, that co-oriented contact manifolds are in one-to-one correspondence with symplectic cones. Given a co-oriented contact manifold $(N,\Ds)$ the symplectic cone over $N$ is given by $(\Ds_+^0,\omega)$ where $\Ds_+^0$ is a connected component of the annihilator of $\Ds_+^0$ devoid of the $0$--section and $\omega$ is the pull-back of the Liouville symplectic form of $T^*N$ via the inclusion $\Ds_+^0\subset T^*N$. There is a natural projection $\pi : \Ds_+^0 \ra N$, so that $\Ds_+^0$ is an $\bR^+$--bundle over $N$. The fibers coincide with the orbits of the vector field, $\tau \in\Gamma(T\Ds_+^0)$, induced by the $\bR^+$--action (i.e multiplication) along the fibers, so that $(\Ds_+^0,\omega)$ is a symplectic cone in the sense that $\mL_\tau \omega = \omega$. Conversely, given a symplectic cone $(M,\omega,\bR^+,\tau)$, then $N:=M/\bR^+$ and $\Ds:= \ker (\tau\intprod \omega)$. 
Note that any contact $1$--form $\eta$ is a section of $\Ds_+^0$ trivializing it as $\Ds_+^0 = N\times \R^+$.  

Similarly, Sasakian manifolds are in one-to-one correspondance with K\"ahler cones. Any Sasakian structure $(g,J, \Phi,\eta,\xi)$ on $(N,\Ds)$ provides a K\"ahlerian structure $(\hat{g},\hat{J})$ on $(\Ds^0_+,\omega)$ as follow: 
 \begin{itemize}
 \item the metric $\hat{g}$ is the unique Riemannian metric which is conic (that is $\mL_\tau \hat{g} = \hat{g}$) and restricts to $g$ on the submanifold $\eta(N) \subset \Ds^0_+$;
 \item the complex structure $\hat{J}$ is the unique complex structure which is homogenous of degree $0$ with respect to $\tau$ (i.e $\mL_\tau \hat{J} =0$) and restricts to $\Phi$ on $\eta(N)$.                                                                                                                                                                                                                                          \end{itemize}
 
%
 
\begin{convention}\label{conventionCONE} Given a manifold $N$ we will denote the topological cone $Y:=N\times \R^+$, $\iota_N: N \ra N\times \{1\}$ the inclusion and $r: Y\ra \R^+$, the projection on the second factor. Any contact $1$--form $\eta$ on $N$ gives rise to a symplectic structure $$\omega := \frac{1}{2}d (r^2\eta)$$ (so that $\iota_N^* \omega = \frac{1}{2}d\eta$) and any compatible Sasaki structure $(J, \xi)$ gives a compatible K\"ahler structure on $Y$ as $$\hat{g}:=dr\otimes dr +r^2g,\;\;\hat{J} := J - \xi\otimes \frac{dr}{r} + r\delr \otimes \eta.$$ Then, by unicity, $(Y,\omega,\hat{g},\hat{J})$ is the K\"ahler cone over $(N,\Ds,J,g,\eta,\xi)$. The compatible K\"ahler structure on $Y$ is then determined by a choice Reeb vector field $\xi\in \gt^+$ giving rise to a {\it polarized} affine variety $(Y,\xi)$.
\end{convention}
%

\subsection{The transverse geometry}\label{subsectransverse}
Most of the important structure of Sasakian geometry arises from its transverse K\"ahlerian structure.

\subsubsection{The transverse structure}
One can consider the class $\Sas(\xi)$ of Sasakian structures having the same Reeb vector field. 
Let $L_\xi$ be the line bundle having $\xi$ as a section. The inclusion $L_{\xi} \hookrightarrow TN$ induces a sequence of bundle morphisms 
$$0\ra L_{\xi} \hookrightarrow TN \ra Q_\xi \ra 0.$$ 
For any CR-structure $(\Ds,J)$ in $\Sas(\xi)$, the restriction followed by projection $\Ds\ra Q_\xi$ is an isomorphism and provides a complex structure $\bar{J}$ on $Q_\xi$. One can consider the subclass of structures $\Sas(\xi,\bar{J})$ making the following diagram commutes.
   \begin{align}
    \begin{split}
       \xymatrix{\relax
           & TN \ar[d]_{\Phi} \ar[r] &  Q_{\xi}\ar[d]_{\bar{J}}\\
         &TN \ar[r] &  Q_{\xi}
       }
   \end{split}
  \end{align} With that comes a natural notion of transversal holomorphic vector fields $\mathfrak{h}(\xi,\bar{J})$, see~\cite{BGS}.  
  
For a given Sasakian manifold, $(N,g,\xi, \Ds,J)$, the transverse K\"ahler geometry refers to the geometry of $(\Ds,J,g|_D)$. More precisely, $N$ is foliated by the Reeb flow. So there are local submersions $\pi_\alpha : U_\alpha \ra V_\alpha$, where $U_\alpha$ and $V_\alpha$ are open subsets of $N$ and $\bC^n$ respectively, such that $\pi_\alpha^* I= \Phi$ where $I$ is the standard complex structure on $\bbc^n$. In particular, $d\pi_\alpha : (\Ds_{|_{U_\alpha}},J) \ra (TV_{\alpha}, I)$ is an isomorphism and the Sasaki metric is sent to a K\"ahler structure $(g_\alpha,\sigma_\alpha,I)$ on $V_\alpha$. 
In the case where $\xi$ is quasi-regular, equivalently, it is induced by the action of a closed $1$--dimensional subtorus $\underline{S}^1_\xi\subset \underline{T}$, the open sets and maps $\pi_\alpha : U_\alpha \ra V_\alpha$ patch up as a global orbifold quotient \begin{equation}\label{quotSYMP}                                                                                                                                                                                                                                                                    \pi: N\longrightarrow W:=N/\underline{S}^1_\xi                                                                                                                                                                                                                                                                         \end{equation}
 endowed with the unique symplectic structure $\sigma$ satisfying $\pi^*\sigma = \frac{1}{2} d\eta$.

\subsubsection{Transverse automorphisms}\label{subsectTautomo} 

\begin{convention}
We will adopt the following point of view, brought from toric geometry. Instead of working with $\underline{T} \subset \mbox{CR}(\Ds,J)$ a Lie group of diffeomorphism of $N$ and $\underline{\kt}$ its Lie algebra of vector fields on $N$, we consider a torus $T=\kt/\Lambda$ where $\kt$ is the Lie algebra and $\Lambda\subset \kt$ a lattice and an injective morphism $\phi: T \hookrightarrow \mbox{CR}(\Ds,J)$. Given $a\in\kt$, the vector field induced on $N$ is $X_a=\phi_*(a)$ or, equivalently, is given at $p\in N$ by $$X_a(p) = \left(\frac{d}{dt} \exp(ta)\cdot p\right)_{t=0}$$ where $\exp : \kt \ra T$ is the quotient map. By maximality, a Reeb vector field $\xi$ is necessarily induced from the action of $T$ and so there exists $b\in \kt$ such that $\xi=X_b$. We denote $\kt^+ =\phi^{-1}(\underline{\kt}^+)$ and call it the (reduced) Sasaki cone as well.  
\end{convention}

%

Let $T$ be a maximal torus embedded in $\mbox{CR}(\Ds,J)$ with Lie algebra $\kt$. For any $T$--invariant compatible couple $(\xi,\eta)$ (i.e. such that $(N,g,\xi, \Ds,J)$ is a Sasakian manifold with $T\stackrel{\phi}{\hookrightarrow}\mbox{Isom}(N,g)$) there is a $\eta$--momentum map $$\mu_\eta : N\longrightarrow \kt^*$$ in the sense that for any $a \in \kt$, we have $$-\frac{1}{2}d\eta(X_a,\cdot) = d\langle \mu_\eta, a\rangle.$$ Such map always exists since one can take $\langle \mu_\eta, a\rangle= \frac{1}{2}\eta(X_a)$ (this is the one we pick in this paper). The image of $\mu_\eta$ lies in the affine subspace $$\mH_b =\{ x\in\kt^* \,|\, \langle x, b\rangle =1/2\}$$ where $b\in\kt^+\subset \kt$ is such that $X_b=\xi$ is the Reeb vector field of $\eta$. Such an element $b$ exists by the maximality of $T$. 

We can consider the infinitesimal $\kt$--action on the open subset $U_\alpha$. Since $\xi$ is induced by the action of $T$, this action descends as an infinitesimal action of $\kt/\bR b$ on $V_\alpha$. It is easy to see that the resulting action should be Hamiltonian. However, the $(\kt/\bR b)$--momentum map is not the map $\overline{\mu}_\eta$ making the following diagram commute    
  \begin{align}
    \begin{split}
       \xymatrix{\relax
           & U_\alpha \ar[d]_{\pi_\alpha} \ar[r]^{\mu_\eta} &  \kt^* \ar[d]^{\mbox{id~.}}\\
         & V_\alpha \ar[r]^{\overline{\mu}_\eta} &  \kt^*
       }
   \end{split}
  \end{align} 
  Indeed the image of this one is not in $(\kt/\bR b)^*$ (i.e $\langle \overline{\mu}_\eta, a+ tb\rangle \neq \langle \overline{\mu}_\eta, a\rangle$ in general). The moment map of the $(\kt/\bR b)$-- action on $(V_\alpha,\sigma_\alpha)$ can be taken to be \begin{equation}\label{quotSYMPmoment}  \check{\mu}_\eta = \overline{\mu}_\eta-x_o\end{equation} for any $x_o \in \mH_b$. Therefore, for any class $[a]\in \kt/\bR b$ 
  $$\langle \check{\mu}_\eta, [a]\rangle = \frac{1}{2}\eta(X_a) - \langle x_o, a\rangle.$$ 
  
 \begin{rem}\label{rem:affine=t*}
  As a direct consequence of this discussion, for any $a\in \kt^*$, we can see $\eta(X_a)$ as a function on $V_\alpha$ and more precisely, we can identify $\kt$ with the set of affine-linear functions on the linear space $\mH_b-\{x_o\}\simeq (\kt/\bR b)^*$, in the sense that $\frac{1}{2}\eta(X_a) = \langle \check{\mu}_\eta, [a]\rangle +\langle x_o, a\rangle$ where the second term is a constant. 
  \end{rem}
  \begin{rem}\label{rem:affine=t_pos=t+}
  If $\xi$ is quasi-regular it induces an $S^1$--action, corresponding to a subgroup $S^1_\xi\subset T$, and we have a global symplectic orbifold quotient $(W= N/\underline{S}^1_\xi, \sigma)$. The image of the momentum map $\Delta_b= \mbox{Im}\check{\mu}_\eta \subset \mH_b-\{x_o\}$ is a compact polytope. Then, the Sasaki cone can be identified with the affine-linear functions on $\mH_b-\{x_o\}\simeq (\mbox{Lie} (T/S^1_\xi))^*$ which are positive on $\Delta_b$, see~\cite{nonUNIQcscS}. 
 \end{rem}

\subsubsection{Transverse curvatures}
The K\"ahler structure $(g_\alpha,\sigma_\alpha,I)$ on $V_\alpha$ has a connection $\nabla_\alpha^T$ and curvatures $R_\alpha^T$, $Ric_\alpha^T$, $\rho^T_\alpha$ $s^T_\alpha$...  Since, $\pi_\alpha^*\nabla_\alpha^T$ and $\pi_\beta^*\nabla_\beta^T$ coincide on $U_\alpha\cap U_\beta$, these objects patch together to define global objects on $N$, the {\it transversal} connection and curvatures $\nabla^T$, $R^T$, $Ric^T$, $\rho^T$, $s^T$... See~\cite{BG:book,FutakiOnoWang} for more details. These tensors are basic, notably the transversal Ricci form $ \rho^T$ satisfies $$ \rho^T(\xi,\cdot) =0, \qquad \mL_\xi \rho^T=0$$ and lies in the basic first Chern class $2\pi c^B_1(\mF_\xi)$.    

Since the exterior derivative preserves this condition, the graded algebra of basic forms is a sub-complex of the de Rham complex. Moreover, one can define the basic exterior derivative $d_B$ as the restriction of the differential to these forms; its adjoint is $\delta_B$ and the basic Laplacian is $\Delta_B=d_B\delta_B+\delta_Bd_B$. The Hodge Theorem holds for the basic cohomology in this context; for a Sasaki metric $g$ there exists a unique basic function $\psi_g$ (of mean value $0$) such that $$\rho^T = \rho^T_H + i\del\overline{\del} \psi_g$$ where $\rho^T_H$ is $\Delta_B$--harmonic. Note that 
$$\Delta_B\psi_g = s^T - \frac{{\bfS}_{\xi}}{\bfV_\xi}=s^T-\bar{s}^T$$ 
where $\bfV_\xi$ is the volume of $(N,g)$ with volume form is 
$$\vol_\xi= \frac{\eta\wedge(d\eta)^n}{n!},$$ 
$\bfS_{\xi}$ is the total transversal scalar curvature, that is 
\begin{equation}\label{TotTransSCAL}
\bfS_{\xi} := \int_N s^T \vol_\xi,
\end{equation} 
and $\bar{s}^T$ is the average transverse scalar curvature.
The volume of the Sasakian manifold $(N,\Ds,J,\xi)$ does not depend on the chosen structure in $\Sas(\xi)$, see~\cite{BG:book} and the total transversal scalar curvature does not depend on the chosen structure in $\Sas(\xi,\bar{J})$, see~\cite{FutakiOnoWang}.

  \section{The Duistermaat--Heckman localization formula and extensions}
  Let $(M, \omega)$ be a symplectic compact manifold admitting a Hamiltonian action of $S^1$ generating the (Hamiltonian) vector field $X\in\Gamma(TM)$ associated to $H : M\ra \R$. In this section, to fit with~\cite{McDuffSalamon}, we take the convention that 
\begin{equation}\label{hamUNUSUALconv} { }_X\intprod\omega =dH.                                                                                                                                                                            \end{equation}   

Let $\bfZ := \mbox{Fix} S^1$ be the fixed points set of $S^1$. It is known that $\bfZ$ consists in a union of smooth symplectic submanifolds $\bfZ=\sqcup_l Z_l$, so the normal bundle $E_Z$ of any component, say $Z \subset\bfZ$, bears a symplectic structure as well and splits into a sum of rank $2$ bundle $E_Z=\oplus_{j} E^Z_j$ according to the action of $S^1$, acting on $E^Z_j$ with weight $\kappa^Z_j$, here and after $j$ runs from $1$ to $n-m_Z$ where $2m_Z= \dim Z$. The Duistermaat--Heckman Theorem~\cite{DuiHeck1,DuiHeck2} says that, under these conditions, we have    
\begin{equation}\label{DHformula} \int_Me^{\omega - H} = \sum_Z e^{- H(Z)}\int_Z \prod_j\frac{e^{\iota_Z^*\omega}}{2\pi c_1(E^Z_j) + \kappa^Z_j}                                                                                                                                                                             \end{equation} where, in the rhs, $Z$ denotes a generic connected component of $\bfZ$, $H(Z)$ the (single) value $H$ is taking on that component and $\iota_Z: Z\hookrightarrow M$ the inclusion. For any form $\psi\in \Omega(M)$, we denote $e^\psi = \sum_{k\geq 0} \frac{\psi^k}{k!}.$

In the simpler case where $H$ is Morse, equivalently, $\bfZ$ is a collection of isolated points, the formula reads 
\begin{equation}\label{DHformulaMorse} \int_Me^{\omega-H}  = \sum_p \frac{e^{- H(p)}}{e(p)}                                                                                                                                                                              \end{equation} where $e(p)$ is the product of the weights at $p\in\bfZ$.

Here, we will use the following extension.  
\begin{theorem}\label{extDHteo} Let $\phi \in \Omega^2_{S^1}(M)$ be a closed form and $f\in C^{\infty}(M)$ be such that ${ }_X\intprod\phi =df$ then \begin{equation}\label{formDHext}
 \int_M(\phi- f)e^{\omega- H} = \sum_Z \int_Z \frac{(\iota_Z^*\phi- f(Z))e^{\iota_Z^*\omega- H(Z)}}{\prod_{j=1}^{n-m_Z} (c_1(E^Z_j) + \kappa^Z_j/2\pi)}.
\end{equation} 
\end{theorem}
\begin{rem}
 The Duistermaat--Heckman formula~\eqref{DHformula}, \eqref{DHformulaMorse} has been explained and generalized in term of equivariant cohomology by Atiyah--Bott \cite{AtiyahBott} and Berline--Vergne~\cite{BerlineVergne}.  The claim~\eqref{formDHext} is a particular case of Theorem 7.13 of \cite[p.216]{BerlineGetzlerVergne}. A mixture of the equivariant cohomology approach and the original proof of Duistermaat--Heckman is given in \cite{McDuffSalamon} when $H$ is a Morse function. 
\end{rem}
%

\begin{rem}\label{remMSYtrick}
Whenever a torus $T$ acts on $(M,\omega)$ in a Hamiltonian fashion and with a momentum map $$\mu: M\longrightarrow \kt^*,$$ it might be interesting to use Formula~\eqref{formDHext} when varying $X=-X_a$, where $a\in \kt$ lies in the lattice $\Lambda$ of circle subgroups. Then $H= \langle\mu, a\rangle$ and we have to replace $\kappa_j^Z= -\langle \underline{\kappa}^Z_j, a\rangle$ for some (fixed) weights $\underline{\kappa}^Z_1,\dots, \underline{\kappa}^Z_{n-m_Z}\in \Lambda^*$. By continuity, Formula~\eqref{formDHext} is valid for any $a\in\kt$ such that $ \langle \underline{\kappa}^Z_j, a\rangle\neq 0$ for $j=1,\dots, n-m_Z$.    
\end{rem}

 \begin{rem}\label{remNonCompact}
 As explained in \cite{MSYvolume} the Duistermaat--Heckman formula~\eqref{DHformula} holds on non-compact manifold assuming that $Z$ lies in the interior of $M$ and the measure tends to $0$ on the boundary of $M$. Indeed, the proof is essentially of local nature, up to the final integration. Moreover, the power series involved in formulas~\eqref{DHformula} and \eqref{formDHext}, are proved to coincide term by term.  Therefore, it holds on non-compact manifold when understood as {\it if the lhs \eqref{formDHext} converges, then it coincides with the rhs of \eqref{formDHext}, which, then, converges}.       
 \end{rem}
 \begin{rem}\label{remOrbi}
 Formula~\eqref{DHformula} also holds for orbifolds, as pointed out in \cite{MSYvolume}, with the slight modification 
 \begin{equation}\label{DHformulaOrbi}
 \int_Me^{\omega - H} = \sum_Z \frac{1}{d_Z}e^{- H(Z)}\int_Z \prod_j\frac{e^{\iota_Z^*\omega}}{(2\pi c_1(E^Z_j) + \kappa^Z_j)/2\pi}                                                                                                                                                                             
 \end{equation} where $d_Z$ is the order of the orbifold group of the generic points of $Z$. Again this orbifold extension is also straightforward for formula~\eqref{formDHext} which becomes 
 \begin{equation}\label{formDHextOrbi}
 \int_M(\phi- f)e^{\omega- H} = \sum_Z \frac{1}{d_Z}\int_Z \frac{(\iota_Z^*\phi- f(Z))e^{\iota_Z^*\omega-  H(Z)}}{\prod_{j=1}^{n-m_Z} (c_1(E^Z_j) +  \kappa^Z_j/2\pi)}
\end{equation}
 \end{rem}

 \begin{rem}\label{rem_omegaDegenerated} Another case of extension is when $\omega$ is degenerated. Indeed, the only property of $\omega$ which is actually important is the relation ${ }_X\intprod\omega =dH$, see~\cite{BerlineGetzlerVergne} together with the nice properties of Hamiltonian group actions which may be ensured by other means than the non-degeneracy of $\omega$.   
   \end{rem}

\section{Applications in Sasakian geometry}

\subsection{An orbi-resolution of the cone}\label{SECTorbiResol}
%
%

Pick $\xi_o\in \kt^+$, a quasi-regular Reeb vector field on $(N,\Ds)$ with symplectic orbifold quotient $(W,\sigma)$. Denote $\cals_o=(N,\Ds, J,\xi_o,\eta_o,g_o)$ the quasi-regular Sasakian structure on $N$ associated to $\xi_o$. Let $b_o\in \kt$ be the lattice element giving the Reeb vector field $\xi_o=X_{b_o}$ and $S^1_o\subset T$ the closed subgroup (a $1$--dimensional torus) induced by $b_o$. We obtain an orbifold resolution of the K\"ahler cone with the cone point included, $Y=Y_0\cup \{0\}$ by constructing the line orbibundle $\mL$ over $W$. We construct it as a K\"ahler reduction as follow. The K\"ahler cone $(Y_0,\omega_o,J_o)$ see \S\ref{sectCONEbackground}, comes equipped with the Hamiltonian action of $T$ and the (homogenous) momentum map $$\hat{\mu} :Y_0 \longrightarrow \kt^*$$ called the universal momentum map by Lerman~\cite{L:contactToric} so that, with respect to the Convention~\ref{conventionCONE}, we have $$\frac{r^2}{2}=\langle\hat{\mu},b_o\rangle.$$ Recall also that we have a natural 
smooth projection 
$\textsf{p} : Y_0\ra N$. \\

On the symplectic product $(Y_0\times \bbc, \omega_o+\omega_{std})$ we consider the Hamiltonian function $$\nu(y,z) = \langle\hat{\mu}_y,b_o\rangle - \frac{|z|^2}{2}$$ with Hamiltonian vector field $\xi_o-\del_\theta$ inducing an $S^1$--action which we denote $S_\cala^1$. One can check that 
\begin{equation*}\begin{split}
\nu^{-1}(1/2) & =\{ (y,z) \,|\, \langle\hat{\mu}_y,b_o\rangle = 1/2 + |z|^2/2\}\\
&\simeq N\times \bbc
\end{split}\end{equation*} and that $\xi_o-\del_\theta$ is nowhere zero on $\nu^{-1}(1/2)$. Therefore $1\in \bR$ is a regular value of $\nu$ and we can perform the symplectic reduction $$\mL = \nu^{-1}(1/2) / S_\cala^1$$ which is an orbifold with a natural symplectic form $\Omega$ defined by $\textsf{q}^*\Omega = \textsf{i}^*(\omega_o+\omega_{std})$ where $\textsf{q}: \nu^{-1}(1/2) \longrightarrow \mL$ is the quotient map and $\textsf{i} :\nu^{-1}(1/2) \hookrightarrow Y_0\times \bbc $ is the inclusion.\\   

A point of $\mL$ is an orbit $[(y,z)]_{\cala} = S_\cala^1\cdot(y,z)$. On the level set $z=0$ the action of $S_\cala^1$ coincides with the one of $S_o^1$ and we get $(\mL \cap \{[(y,z)]_{\cala}\,|\,z=0\})\simeq  W:= \{ [\textsf{p}(y)]_o=  S_o^1\cdot \textsf{p}(y)\}$.  
 The space $\mL$ is a line bundle over $W$ with bundle map $\pi([(y,z)]_{\cala}) =  [\textsf{p}(y)]_o$. Moreover, there is a smooth bijective map \begin{equation}\label{biholomLY}
 \begin{array}{cccc}
            f:& \mL_0  &\longrightarrow  &Y_0  \\
        & [(y,z)]_{\cala} &\mapsto  & ((z/|z|)\cdot_o \textsf{p}(y), |z|) \\
 \end{array}
\end{equation}
  where $\cdot_o$ denotes the action of $S^1$ induced by $\xi_o$ (that is $e^{2\pi i \theta}\cdot_o x= \exp_T(\theta b_o)\cdot x$) and $\mL_0 = \{[(y,z)]_{\cala}\in \mL \,|\,z\neq 0\}=\mL\backslash$ ($0-$section). Observe that the pull-back $\textsf{q}^*f :Y_0\times \bbc \longrightarrow Y_0$ is $\bbc^*$--invariant and one can check by direct computation that $(\textsf{q}^*f)_*\circ (J_o+i) = J_o\circ (\textsf{q}^*f)_*$. Hence, $f : \mL_0\ra Y_0$ is biholomorphic and we get a birational map $f: \mL  \dashrightarrow  Y_0$.

Note that the total space of $\mL$ inherits the Hamiltonian action of $G=(T\times S^1)/S_\cala^1$ with momentum map $$\check{\nu}: \mL\longrightarrow \kg^*\simeq \{(\alpha,t)\in \kt^*\oplus \bbr \,|\, \langle\alpha,b_o\rangle = 1/2+t/2\}.$$ Given $u\in T$ we may consider the element $[(u, 1)]_{S_\cala^1}\in G$, this gives an isomorphism making $f$ an equivariant map. In what follows we identify the action of $G$ and $T$ using that latter map. The image of $\check{\nu}$ is identified with the truncated polyhedral cone $\{x\in \im \hat{\mu}\subset \kt^*\,|\, \langle x,b_o \rangle \geq 1/2\}$. 

  The zero section $W\hookrightarrow \mL$, with its induced symplectic structure is naturally identified with the symplectic quotient $(W,\sigma)$, see~\eqref{quotSYMP}, associated to the quasi-rational Reeb vector field $\xi_o=X_{b_o}$. Also the action of $G$ coincides with the action of $T/S^1_{o}$ with momentum map $\check{\mu}_{\eta_o}: W\ra (\kt/\bR b_o)^*$ see~\S\ref{subsectTautomo}. 
  
  \subsubsection{The weights of the torus action on the orbi-resolution}\label{secWeights} 
  Let $Z\subset \mL$ be a connected component of the fixed points set of $T$. In particular $Z\subset W$ is seen inside $\mL$ as the zero section. For $p\in Z$, the action of $T$ gives a decomposition into equivariant bundles 
\begin{equation}\label{decomptangentT}T_p\mL = T_pW \oplus \mL_p =T_pZ \oplus \left(\oplus_{j=0}^{n-m_Z} E_{j,p}^Z\right).\end{equation} 
Here, and in what follows, $\mL_p$ is canonically identified with its tangent in $T\mL$.   
  
  Since $Z\subset W$ the fiber $\mL_p$ is one of the summand in the lhs of~\eqref{decomptangentT} and we put $$\mL_p =E_{0,p}^Z.$$ We can easily compute the weight of the action of $T$ on the fiber $\mL_p$. Indeed, $p$ is fixed by $T$ if and only if for any $y\in Y_0$ such that $[\textsf{p}(y)]_o = p$ we have $\{X_a(y) \,|\,a\in\kt\} =\bR \xi_o(y)$. Then, for such a point $p= [(y,0)]_{\mA}$ and $b \in \kt$ we can define $\delta_b\in\bR$ by $\delta_bX_{b_o}(y)=X_b(y)$ so that $$[( \exp_T(tb) y, z)]_{\mA}= [( \exp_T(tb -t\delta_bb_o) \cdot y, e^{t\delta_b} z)]_{\mA}= [(y, e^{t\delta_b} z)]_{\mA}.$$ Therefore, the weight $\underline{\kappa}_{0}^Z\in\kt^*$ is 
           \begin{equation}
            \langle\underline{\kappa}_{0}^Z, b \rangle = \delta_b= \eta_o(X_b)_q 
           \end{equation} for any $q\in N \cap \pi^{-1}(p)$.  

Note that the other weights $\underline{\kappa}_1^Z,\dots, \underline{\kappa}_{n-m_Z}^Z \in \kt^*$ all lie in $(\kt/\bR b_o)^*\simeq b_o^0 \subset \kt^*$, the annihilator of $b_o$, since $ \oplus_{j=0}^{n-m_Z} E_{j,p}^Z \subset T_pW$ and $W\subset \mbox{Fix }_\mL S^1_o$.

\subsection{The volume functional}\label{sectAppSasVol}\label{sectMSY}

To prove Theorem~\ref{theoVol=rat} in~\cite{MSYvolume} when the K\"ahler cone is Gorenstein, Martelli, Sparks and Yau suggest to use the Duistermaat--Heckman formula~\eqref{DHformula}, proved in~\cite{DuiHeck1,DuiHeck2}, on the crepant resolution of the cone (with its cone point), assuming the existence of an approximation of the pulled back cone metric on the resolution. Beside, they produce an algebro-geometric proof of their Theorem. Our proof follows the lines of the first suggested approach in \cite{MSYvolume}, but we get rid of the assumption on the existence of the approximation, use an orbi-line bundle and we don't need the assumption that the K\"ahler cone is Gorenstein. \\

Let $(N,\Ds, J g,\xi,\eta)$ be a Sasakian manifold of dimension $2n+1$ and $(Y_0,\omega,J)$ its K\"ahler cone, see Convention~\ref{conventionCONE}. We denote $Y= Y_0 \cup \{0\}$ the union of $Y_0$ with its cone point. Note that, with respect to the notation of~\S\ref{sectBackground}, the Hamiltonian function of $\xi$ is $r^2/2$, in the sense that $\omega(\xi,\cdot) =- d(r^2/2)$, since $\omega =\frac{1}{2} d(r^2\eta)$. Moreover, using Fubini's theorem and integration by parts, we have
\begin{equation}\label{volumeformulaIndex}
\bfV_{\xi} := \int_N \frac{\eta\wedge(d\eta)^n}{n!} = \frac{a^{n+1}}{n!} \int_{Y_0} e^{-a r^2/2} \frac{\omega^{n+1}}{(n+1)!} 
\end{equation} for any $a>0$.
The natural idea is to use the Duistermaat--Heckman formula to expand the rhs term as a rational function of $\xi\in \kt^+$; however, $\xi$ has no fixed point. An idea of Martelli--Sparks--Yau is then to use a resolution of $Y$ and approximate the K\"ahler structure of $(Y_0,\omega,J)$  on it.

We will use the Duistermaat-Heckman formula~\eqref{DHformulaOrbi}, Remark~\eqref{remMSYtrick}of Martelli--Sparks--Yau   and the observation~\eqref{volumeformulaIndex} on the total space of $\mL$ constructed in \S\ref{SECTorbiResol}. Observe that in these formulas $\omega$ is not required to be symplectic and $\eta$ does not need to be defined everywhere. Namely we note that 
\begin{itemize}
 \item[(i)] equation \eqref{volumeformulaIndex} holds if and only if $Y_0=N\times \R^+$ and $$r^2\in C^{\infty}(Y_0),\; \omega = \frac{1}{2}d(r^2\eta) \in\Omega^2(Y_0), \; \eta \mbox{ is invariant by dilatation};$$
 \item[(ii)] equation \eqref{DHformulaOrbi} holds if and only if $\omega$ is closed and $\omega(X,\cdot) = dH$;
 \item[(iii)] \eqref{DHformulaOrbi} coincides with \eqref{volumeformulaIndex} (for $a=1$) if and only $-r^2/2=-H$.
\end{itemize}
 
Observe that the restriction of the function $z\ra |z|$ on $\nu^{-1}(1/2)\subset Y_0\times \bbc$ (see \S\ref{SECTorbiResol}) is well-defined on the quotient  $\mL$ and that $|z|^2$ is smooth. Moreover, $|z|=f^*r$ where the map $f$ is defined by \eqref{biholomLY} and $r$ is the projection on the second factor of $Y_0$, see Convention~\ref{conventionCONE}. 

The vector field $\xi=X_b$, as a vector field on $Y_0$, commutes with the action of $\xi_o$, thus of $\mA$, and descends as a well-defined vector field, still denoted $\xi=X_{b}$, on $\mL$. The contact form $\eta_\xi$ on $N$, associated to $\xi$, is not basic with respect to $\xi_o$ and thus, not well-defined on $\mL$ a priori. However, we can {\it define} $\tilde{\eta}$ on $\mL\backslash W$ by imposing $\tilde{\eta}(\xi) =1$, $d\tilde{\eta}(\xi,\cdot) =0$ and that $\tilde{\eta}$ coincides with $f^*\eta$ on $r^{-1}(1)$. It does not define $\tilde{\eta}$ uniquely but we only need these conditions to make our argument work.  
Note that $\xi$ commutes with the vector field $r\frac{\del}{\del r}$ induced by the dilatation on the second factor of $N\times \bbc$. Hence, $\tilde{\eta}$ is invariant by the dilatation, that is homogenous of degree $0$ with respect to  $r\frac{\del}{\del r}$. Then, using observation (i) above, \eqref{volumeformulaIndex} holds on $Y_0\simeq\mL\backslash W$ for $\omega := \frac{1}{2}d(r^2 \tilde{\eta})$. 

Now we can forget $\tilde{\eta}$ and work with the rhs of \eqref{volumeformulaIndex} for $a=1$. Indeed, the $1$--form $\tilde{\eta}$ cannot be extended continously on $\mL$ but $r^2\tilde{\eta}$ can and then $\omega$ is a smooth closed $2$--form on $\mL$. So, the integrand of the rhs of \eqref{volumeformulaIndex} makes sense on $\mL$ and then, because $W$ has measure zero, we have   
\begin{equation}\label{volumeformulaIndexL}
\int_N \frac{\eta\wedge(d\eta)^n}{n!} = \frac{1}{n!} \int_{\mL} e^{-r^2/2} \frac{\omega^{n+1}}{(n+1)!}. 
\end{equation} Here, we identify $N$ with $r^{-1}(1) \subset \mL$. 

We can use \eqref{DHformulaOrbi} with $\omega$, $X=-\xi$ and $H=r^2/2$ which are defined on the whole $\mL$. The equation $\omega(X,\cdot)= dr^2 /2$ holds on $W$, because both sides vanish, and on $\mL\backslash W$, thanks to the fact that $\tilde{\eta}(X) = -1$, $d\tilde{\eta}(X,\cdot) =0$ and $\omega = \frac{1}{2}d(r^2 \tilde{\eta})$. Hence, using observations (ii) and (iii) above we get that \eqref{DHformulaOrbi} coincides with \eqref{volumeformulaIndexL}. That is  
\begin{equation}\label{DHformulaOrbiL}\begin{split}
 \int_{\mL} e^{-r^2/2}\frac{\omega^{n+1}}{(n+1)!} &=  \sum_Z \frac{1}{d_Z}\int_Z \prod_j\frac{e^{\iota_Z^*\omega}}{c_1(E^Z_j) + \kappa_j^Z/2\pi}\\
 &= \sum_Z \frac{1}{d_Z}\int_Z \prod_j\frac{e^{\iota_Z^*\omega}}{c_1(E^Z_j) - \langle \underline{\kappa}_j^Z, b\rangle/2\pi } 
 \end{split}
\end{equation} where the last line comes from Remark~\ref{remMSYtrick} to highlight the dependance on $b\in \kt^+$.

In the formula \eqref{DHformulaOrbiL}, the fixed point set $\bfZ=\cup Z$ of $X$ lies in $W$ on which $\omega$ vanishes identically. Thus, in the integrand of the rhs of \eqref{DHformulaOrbiL}, only the first term of the power series in $\iota_Z^*\omega$ does not vanish a priori. In conclusion, we have 
\begin{equation}\label{formulaDHrationalVolL}
\bfV_{X_b} =  \frac{1}{n!} \sum_Z \frac{(2\pi)^{n-m_Z}}{d_Z}\int_Z \prod_{j=0}^{n-m_Z}\frac{1}{2\pi c_1(E^Z_j) - \langle \underline{\kappa}_j^Z, b\rangle} 
\end{equation} which can be written as a rational function of $\langle \underline{\kappa}_j^Z, b\rangle$'s using the relation $(1-x)^{-1} = \sum_{s\geq 0} x^s$. Since the integration over $Z$ picks up only the term of degree $2m_Z$, the factor $(2\pi)^{m_Z}$ goes out and we get a rational function of order $-(n+1)$ up to the over all factor $(2\pi)^{n+1}$ 
\begin{equation}\label{formulaDHrationalVolL}
\bfV_{X_b} =  \frac{(2\pi)^{n+1}}{n!} \sum_Z \frac{1}{d_Z}\int_Z \prod_{j=0}^{n-m_Z} \frac{1}{\langle \underline{\kappa}_j^Z, b\rangle} \left(  \sum_{s\geq 0} \left(\frac{c_1(E^Z_j)}{\langle \underline{\kappa}_j^Z, b\rangle}\right)^s\right).
\end{equation}
This concludes the proof of Theorem \ref{theoVol=rat} \hfill $\Box$

\begin{rem} If we start with a generic $b\in\kt^+$, \eqref{formulaDHrationalVolL} holds for the dense open subset of $\kt^+$ of elements sharing the same fixed points set $\{Z\}\subset W$ which coincides with the fixed points of $T/S_{\xi_o}^1$, the maximal torus acting on $W$. For example, in the toric case starting with a generic $b\in\kt^+$, then $\{Z\}$ is the fixed points set of the torus and so consists of a finite number of isolated points. In this case, formula \eqref{formulaDHrationalVolL} gives back the formulas in ~\cite{reebMETRIC, nonUNIQcscS} as expected.  
\end{rem}

 \begin{corollary} \label{coroVboundary}When $b\in \kt^+$ tends to the boundary of $\kt^+$, $\bfV_{X_b}$ tends to $+\infty$.  
 \end{corollary}
 \begin{proof} When $b\in \kt^+$ tends to the boundary of $\kt^+$ (but away from $0\in \del \kt^+$), there is at least one point $p\in N$ such that $\eta_o(X_b)_p \ra 0$ by definition of $\kt^+$, see \eqref{eq:SasakiCone}. Therefore, $p$ is a minimum of the function $\eta_o(X_b) =\langle\mu_{\eta_o},b \rangle$ and $\mu_{\eta_o}(p)$ lies in the boundary of the polytope $\Delta_o=\mu_{\eta_o}(N)$ which, up to translation is the moment polytope of $(W,\sigma,T/S^{1}_o)$. 
 
 Recall from \S\ref{subsectTautomo}, that the function $\eta_o(X_b)$ is the pullback to $N$ of an affine-linear function on $\Delta_o$ so that it is a Morse-Bott function on $N$ of even index. By a well known argument see~\cite{GS}, $\eta_o(X_b)$ has a unique minimum in the sense that there is a unique connected component $\tilde{Z}\subset \mbox{crit}(\eta_o(X_b)) \subset N$ such that $\tilde{Z}$ is a local minimum, thus it is a global minimum and there is no other part of $N$ on which $\eta_o(X_b)$ tends to $0$ faster than it does on $\tilde{Z}$.  
 
Moreover, near $Z= \pi(\tilde{Z})\subset W$ and seen as a function on $W$ it can be locally written as $\eta_o(X_b) =_{loc} \sum_{i=1}^{n-m_Z} |v_i|^2 \langle \underline{\kappa}_i^Z , b\rangle + c_b$ where $c_b$ is a constant and $v_i \in E^Z_i$ in the decomposition \eqref{decomptangentT} corresponding to the action of $T/S^1_o$ endowed with an invariant and compatible metric $|\cdot|$, see e.g.~\cite{L:contactToric}. Since $p$ is a minimum we have $$\langle \underline{\kappa}_i^Z , b\rangle >0$$ for $i=1,\dots , n-m_Z$. Therefore, combining this with the observations of \S\ref{secWeights}, in particular that $c_1(E_0^Z) = \iota_Z^*c_1(\mL)= \iota_Z^*[\sigma/2\pi] >0$ and $\langle \underline{\kappa}_0^Z , b\rangle = \eta(X_b)_p$ we get that the leading term of $\bfV_{X_b}$, which is 
 \begin{equation}\label{leadVol}
\frac{(2\pi)^{n+1}}{n!d_Z} \left(\prod_{j=0}^{n-m_Z} \frac{1}{\langle \underline{\kappa}_j^Z, b\rangle} \right)\int_Z\left(\frac{c_1(E^Z_0)}{\langle \underline{\kappa}_0^Z, b\rangle}\right)^{m_Z},
\end{equation} is positive and tends to $+\infty$ when $\langle \underline{\kappa}_0^Z , b\rangle=\eta_o(X_b)_p \ra 0$.\end{proof}

\subsection{The total transversal scalar curvature functional}\label{sectTTSCf}

 \subsubsection{Variation of the transverse scalar curvature}

The goal is to understand better how the function $\xi \mapsto \bfS_{\xi}$ varies with $\xi$ when $\Ds$ is a fixed (co-oriented) contact structure with a compatible CR structure $(\Ds,J)$ on $N$, see~\eqref{defnSasakiCONE}.

Given a Sasaki structure $(N,\Ds,J,\xi,\eta_\xi)$ on $N$, we recall from~\cite{FutakiOnoWang} that   
\begin{equation}\label{eq:DefnbfSrecallLOC}
s^T \frac{\eta_\xi \wedge (d\eta_\xi)^{n}}{n!} = 2\frac{\rho_\xi^T \wedge\eta_\xi\wedge  (d\eta_\xi)^{n-1}}{(n-1)!}. 
\end{equation} and $\rho_\xi^T$ is defined on each patch $U_\alpha \subset N$ as the pull back of the Ricci form $\rho_{\xi,\alpha}$ of the K\"ahler structure $(g_{\xi,\alpha},\omega_{\xi,\alpha},I)$ induced on $V_\alpha$ (see the notation from the second paragraph of~\S\ref{subsectransverse}, taken from~\cite{FutakiOnoWang}, here we add the indice $\xi$ to emphasize the dependance in the structure).  

Consider another Sasaki structure $(N,\Ds,J,\xi_o,\eta_o)$ compatible with the same CR structure $(\Ds,J)$ and $\xi_o\in\kt^+$. The associated transversal Ricci form $\rho_{o}^T \in \Gamma(N,\bigwedge^2 \Ds)$ is not basic with respect to $\xi$. Indeed $\mL_\xi\rho_{o}^T\equiv 0$ by $T$--invariance of the Sasaki structure $(N,\Ds,J,\xi_o,\eta_o)$, but there is no reason for $\rho_{o}^T(\xi,\cdot)$ to vanish identically. Actually, $\xi_o$ commutes with $\xi$ by hypothesis and descends as a Killing vector field on $(V_\alpha,g_{\xi,\alpha},\omega_{\xi,\alpha},I)$. Therefore by the well-known formula, see~\cite[Lemma 1.23.4]{PGnote}, we have that 
\begin{equation}\label{eq:rho_notbasic}
 \rho_{\xi,\alpha}(\xi_o,\cdot) = -\frac{1}{2} d \Delta^{g_{\xi,\alpha}} \langle \check{\mu}_\xi,[b_o]\rangle
\end{equation} where $\check{\mu}_\xi : V_\alpha \ra (\kt/\bR b)^*$ is the momentum map and $b_o$ is such that $\xi_o=X_{b_o}$ (see~\S\ref{subsectTautomo}).

Up to a constant, the pull back by $\pi_{\xi,\alpha}$ of the function $\langle \check{\mu}_\xi,[b_o]\rangle$ coincides with $ \langle \mu_\xi,b_o\rangle= \frac{1}{2}\eta_\xi(\xi_o)$ which is globally defined on $N$. For any function $f\in C^2(V_\alpha)$ we have $\pi_{\xi,\alpha}^* \Delta^{g_{\xi,\alpha}} f = \Delta^{g_{\xi}} \pi_{\xi,\alpha}^*f$ because $\pi$ is a Riemannian fibration with totally geodesic fibers~\cite{watson}. Therefore, on $N$, we have the global formula
\begin{equation*}
 \rho_{\xi}^T(\xi_o,\cdot) = -\frac{1}{2} d\Delta^{g_{\xi}} \langle \mu_\xi,b_o\rangle, 
\end{equation*} or similarly
\begin{equation}\label{eq:rho_onotbasic}
 \rho_{o}^T(\xi,\cdot) = -\frac{1}{2} d\Delta^{g_{o}} \langle \mu_{\xi_o},b\rangle. 
\end{equation} In particular, putting $f_{o,\xi}:=\frac{1}{2} \Delta^{g_{o}} \langle \mu_{\xi_o},b\rangle$  we get that the $2$--form $$\beta =  \rho_{o}^T +\eta_\xi\wedge df_{o,\xi} - f_{o,\xi}d\eta_\xi = \rho_{o}^T - d(f_{o,\xi}\eta_\xi)$$ is closed and $\xi$--basic (i.e basic with respect to $\xi$). So the two closed $2$--forms $\rho_\xi^T$ and $\beta$ lie in the same de Rham cohomology class, namely $c_1(\Ds)$, and are both $\xi$--basic and $J$--invariant. Therefore, there exists a $\xi$--basic function $f\in C^\infty_B(N)$ such that    
\begin{equation}\label{eq:beta} \rho_{\xi}^T -\beta = d_Bd_B^c f. \end{equation}

Note that any basic function $f\in C_B^2(N)$ defines a function, still denoted $f$, on each chart $V_\alpha$ and, applying a classical K\"ahler formula on $(V_\alpha, g_{\xi,\alpha},\omega_{\xi,\alpha},I)$ (see for e.g~\cite{PGnote}), we have 
$$\frac{dd^c f \wedge (\omega_\xi^T)^{n-1}}{(n-1)!} = -(\Delta^{g_{\xi,\alpha}} f) \, \frac{(\omega_\xi^T)^{n}}{n!}$$ where $\Delta^{g_{\xi,\alpha}}$ is the Laplacian associated to the metric $g_{\xi,\alpha}$. Therefore, we have for any basic function  $f\in C_B^2(N)$ $$\frac{d_Bd_B^c f \wedge\eta_\xi \wedge (d\eta_\xi)^{n-1}}{(n-1)!} = -(\Delta^{g_{\xi}}_B f) \, \frac{\eta_\xi \wedge (d\eta_\xi)^{n}}{n!}$$ where $\Delta^{g_{\xi}}_B f = d_B^*d_B f$ is the basic Laplacian.

Putting this, \eqref{eq:beta} and~\eqref{eq:DefnbfSrecallLOC} in \eqref{TotTransSCAL} and using the fact that $\eta_\xi\wedge \eta_\xi\equiv 0$, we get 

\begin{equation}\label{eqSCALbeta}\begin{split}
\bfS_{\xi} &=2\int_N \frac{\rho_\xi^T \wedge\eta_\xi\wedge  (d\eta_\xi)^{n-1}}{(n-1)!}\\
&= 2\int_N \frac{\beta \wedge\eta_\xi\wedge  (d\eta_\xi)^{n-1}}{(n-1)!} + \int_N\frac{(d_B d_B^c f)\wedge  \eta_\xi \wedge (d\eta_\xi)^{n-1}}{(n-1)!}.\\
&= 2\int_N \frac{\beta \wedge\eta_\xi\wedge  (d\eta_\xi)^{n-1}}{(n-1)!} - \int_N(\Delta^{g_{\xi}}_B f) \frac{\eta_\xi \wedge (d\eta_\xi)^{n}}{n!}.\\
&= 2\int_N \frac{\beta \wedge\eta_\xi\wedge  (d\eta_\xi)^{n-1}}{(n-1)!}\\
&= 2\int_N \frac{(\rho_o^T - d(f_{o,\xi}\eta_\xi)) \wedge\eta_\xi\wedge  (d\eta_\xi)^{n-1}}{(n-1)!}\\
&= 2\int_N \frac{\rho_o^T \wedge\eta_\xi\wedge  (d\eta_\xi)^{n-1}}{(n-1)!} - 2\int_N f_{o,\xi} \frac{\eta_\xi\wedge  (d\eta_\xi)^{n}}{(n-1)!}
\end{split} 
\end{equation} Note that $f_{o,\xi}=\frac{1}{2} \Delta^{g_{o}} \langle \mu_{\xi_o},b\rangle =\frac{1}{4} \Delta^{g_{o}} \eta_o(\xi)$ is linear with respect to $\xi\in \kt^+$. Putting $\vol_o=\frac{\eta_o \wedge (d\eta_o)^{n}}{n!}$, another way to write \eqref{eqSCALbeta} is 

\begin{equation}\label{eqSCALT} \begin{split}
\bfS_{\xi} &=  \int_N \frac{s_o^T}{\eta_o(\xi)^n} \vol_o - 2\int_N \frac{f_{o,\xi}}{\eta_o(\xi)^{n+1}} \vol_o\\
&=  \int_N \frac{s_o^T}{\eta_o(\xi)^n} \vol_o + \frac{n+1}{2}\int_N \frac{|d\eta_o(\xi)|^2_{g_o}}{\eta_o(\xi)^{n+2}} \vol_o
\end{split}
\end{equation} 
\begin{corollary} Within a Sasaki cone there is at most one ray of vanishing transverse scalar curvature. Moreover, if there is one Sasaki structure with non negative transverse scalar curvature then the total transverse scalar curvature is non-negative on the whole Sasaki cone. 
\end{corollary}

\subsubsection{A localization formula for the total transverse scalar curvature}
Using polar coordinates on $Y_0= N\times \bR_{>0}$, Fubini's Theorem and integration by parts, we have for any $\eta\in \Omega^1(N)$, $\phi\in \Omega^2(N)$ and $a>0$ 
\begin{equation}\label{eqTotSonY}
\int_N \frac{\eta\wedge \phi \wedge (d\eta)^{n-1}}{(n-1)!} = \frac{a^{n} }{(n-1)!}\int_{Y_0} e^{- a r^2/2} \frac{\phi \wedge \omega^{n}}{n!}  
\end{equation} where $\omega =\frac{1}{2} d r^2 \eta$. 

Given a Sasakian  manifold $(N,\Ds,J,g,\xi,\eta)$ of dimension $2n+1$ as above, we use the notation of \S\ref{sectAppSasVol} and \S~\ref{sectBackground}. We consider the resolution $f:\mL\dashrightarrow Y$ of section \ref{SECTorbiResol} associated with a quasi-regular Reeb vector field $\xi_o \in\underline{\kt}^+$, inducing the isometric action of a circle $S^1_{\xi_o}\subset \underline{T}$ whose quotient K\"ahler orbifold is $(W,\sigma_o,\check{J}_o,\check{g}_o)$. Recall that $\mL$ is obtained as the quotient of $\nu^{-1}(1/2)=N\times \C$ by the action of $S^1$ induced by $\xi_o -\del_\theta$. In particular, a $\xi_o$ basic $k$--form on $N$ descends as a well defined $k$--form on $\mL$, for example $\rho_o^T$ and $f_{o,\xi}$ are both defined on the total space of $\mL$. Then, we can use the observations~\eqref{eqTotSonY} and~\eqref{volumeformulaIndex} to modify respectively the first and the second term of the last line of~\eqref{eqSCALbeta} as follow  
\begin{equation}\label{eqSCALY}\begin{split}
\frac{(n-1)!}{2}\bfS_{\xi} &= \int_N \rho_o^T \wedge\eta_\xi\wedge  (d\eta_\xi)^{n-1} - \int_N f_{o,\xi} \eta_\xi\wedge  (d\eta_\xi)^{n}\\
&=  \int_{Y_0} e^{-r^2/2} \frac{\rho_o^T \wedge \omega^{n}}{n!} -  \int_{Y_0} f_{o,\xi} e^{-r^2/2} \frac{\omega^{n+1}}{(n+1)!} \\
&= \int_{Y} e^{-r^2/2} (\rho_o^T-f_{o,\xi}) e^{\omega}\\
&= \int_{Y}  (\rho_o^T-f_{o,\xi}) e^{\omega-r^2/2}
\end{split} 
\end{equation} where $\omega =\frac{1}{2} d r^2 \eta_\xi$. 

Using Theorem~\ref{extDHteo} in the orbifold case (see Remark \ref{remOrbi}) with $X=-\xi$, $H=r^2/2$, we get 
\begin{equation}\label{eqSCALsumZ}\begin{split}
\frac{(n-1)!}{2}\bfS_{\xi} &= \sum_Z \frac{1}{d_Z}\int_Z \frac{(\iota_Z^*\rho^T_o - f_{o,\xi}(Z))e^{\iota_Z^*\omega}}{\prod_{j=0}^{n-m_Z} ( c_1(E^Z_j) +  \kappa^Z_j/2\pi)} \\
&= \sum_Z \frac{1}{d_Z}\int_Z \frac{(\iota_Z^*\rho^T_o- f_{o,\xi}(Z))e^{\iota_Z^*\omega}}{\prod_{j=0}^{n-m_Z}(c_1(E^Z_j) - \langle   
\underline{\kappa}_j^Z,b\rangle/2\pi )}. 
\end{split} 
\end{equation} Again, the fixed points set $\bfZ=\cup Z$ of $X$ lies in $W$ on which $\omega$ vanishes identically. Thus, in the integrand of the rhs of \eqref{eqSCALsumZ}, only the first term of the power series in $\iota_Z^*\omega$ does not vanish a priori. In conclusion, we have 
\begin{equation}\label{eqSCALsumZ0}\begin{split}
\bfS_{\xi} &= \frac{2}{(n-1)!}\sum_Z \frac{1}{d_Z}\int_Z \frac{(\iota_Z^*\rho^T_o- f_{o,\xi}(Z))}{\prod_{j=0}^{n-m_Z} ( c_1(E^Z_j) -  \langle \underline{\kappa}_j^Z,b\rangle/2\pi )}\\
 &= \frac{2 (2\pi)^{n+1}}{(n-1)!} \sum_Z \frac{1}{d_Z} \left[  \int_Z (\iota_Z^*\rho^T_o/2\pi)\wedge \left(\prod_{j=0}^{n-m_Z} \frac{1}{\langle \underline{\kappa}_j^Z, b\rangle}\sum_{s\geq 0} \left(\frac{c_1(E^Z_j)}{\langle \underline{\kappa}_j^Z, b\rangle}\right)^s\right) \right. \\
 &\qquad \qquad \qquad\qquad\qquad \left. -f_{o,\xi}(Z)\int_Z \prod_{j=0}^{n-m_Z} \frac{1}{\langle \underline{\kappa}_j^Z, b\rangle}\sum_{s\geq 0} \left(\frac{c_1(E^Z_j)}{\langle \underline{\kappa}_j^Z, b\rangle}\right)^s\right]
\end{split} 
\end{equation} 

\begin{rem} Whenever the fixed points set $\bfZ$ of $X$ consists in a set of isolated points $\bfZ=\{p\}$, formula~\eqref{eqSCALsumZ0} simplifies as  
\begin{equation}\label{eqSCALsum_p}
\bfS_{\xi} = \frac{2(2\pi)^{n+1}}{(n-1)!}\sum_p \frac{1}{d_p}  \frac{-f_{o,\xi}(p)}{\prod_{j=1}^{n+1}\langle \underline{\kappa}_j^p,b\rangle }.  
\end{equation}  
\end{rem}

\subsubsection{Proof of Theorem~\ref{theoSCAL=rat}} \label{subsecProoftheoSCAL=rat}

Using \eqref{eqSCALsumZ0} and the facts that $\frac{\iota_Z^*\rho^T_o}{2\pi},c_1(E^Z_j)  \in H^2_{dR}(Z,\bQ)$ and $\lambda_j^Z\in \Lambda^*$ (see Remark~\ref{remMSYtrick}), to prove Theorem~\ref{theoSCAL=rat}, it remains to show that $f_{o,\xi}(Z) \in \bQ$.

\begin{lemma}\label{lemFIXlap} Let $(M^n,\omega, J,g)$ be a K\"ahler orbifold endowed with an effective Hamiltonian isometric action of a torus $T^m$ and a momentum map $\mu : M\ra \kt^*$ (where $\kt:=\mbox{Lie}(T)$). For any $Z$, a connected component of $\mbox{Fix}_M T$ with $\dim Z =2m_Z$, denoting the equivariant decomposition $T_zM = T_zZ \oplus_{j=1}^{n-m_Z} E_j^Z$ with weights  $\underline{\kappa}^Z_1,\dots,  \underline{\kappa}^Z_{n-m_Z}$ we have $$(\Delta^{g} \mu)_z =-2\sum_{j=1}^{n-m_Z} \underline{\kappa}^Z_j$$ at any $z\in Z$.
\end{lemma}
\begin{proof}
We prove the lemma for $M$ a manifold. Since the result is local this assumption is not restrictive. 

Let first recall some consequences of the Symplectic Slice Theorem, see for eg.~\cite{Audin,LT}. We choose $z\in M$, a fixed point of $T$ and call $\phi : T\ra \mbox{Symp}(T_xM,\omega)$, the representation so that $\phi_{\theta}=d_z\theta$ where $\theta\in T$ is viewed as a diffeomorphism of $M$; we write the corresponding representation $\phi_i : T\ra \mbox{Symp}(E_i^Z, \omega_i)$. With this notation, $\underline{\kappa}^Z_i=d\phi_i\in\kt^*$ are the weights of the action $\phi$. Note that $W_0$ is the tangent space of the connected component $Z$ of the fixed points set of $T$.

Let $b\in \kt$ and put $\mu_b=\langle \mu , b\rangle$. The Laplacian $\Delta^{g}\mu_b$ is also $-\mbox{tr Hess}^{g} \mu_b$ where we have at $z$, for vector fields $X$, $Y$ defined near $z$: 
\begin{equation}
 \begin{split}
  \mbox{Hess}^{g} \mu_b (X,Y) &=\nabla d\mu_b (X,Y)\\
  &=  X (d\mu_b(Y))  \\
  &=- X(\omega(X_b,Y))\\
  &= -\omega(\nabla_X X_b, Y)
 \end{split}
\end{equation} by using twice the fact that $z$ is a critical point of $\mu_b$ and the K\"ahler condition $\nabla J=0$. Hence, since $X_b$ is a real holomorphic vector field, that is $\nabla_{JY}X_b=J\nabla_{Y}X_b$, and that $\mL_X\omega(X_b,Y) =0$ for any $X\in\Gamma(TZ)$, we obtain $$\Delta^{g}\mu_b =  \sum_{i=1}^{n}\omega(\nabla_{w_i} X_b, w_i)+\omega(\nabla_{J w_i} X_b, Jw_i)$$ where $\{w_i, Jw_i\}$ is a normal symplectic basis of $E_i$ for $i=1,\dots,n-m_Z$ and $\{w_i, Jw_i\}_{i=n-m+1}^n$ is a normal symplectic basis of $T_zZ$. Thanks to the fact that $X_b$ is real holomorphic, it reduces 
$$\Delta^{g}\mu_b =  2\sum_{i=1}^{n}\omega(\nabla_{w_i} X_b, w_i).$$
Consider the path $\theta_t= \mbox{exp }t b$ in the torus $T$, so $(d\phi) (b) =\frac{d}{dt}_{|_{t=0}}\phi_{\theta_t}\in \mathfrak{symp}(T_zM, \omega)$ and $v\in T_zM$ with the flow $\gamma_s\subset M$ (i.e $\frac{d}{ds}_{|_{s=0}} \gamma_s =v$).  With this notation 
\begin{equation}
 \begin{split} d\phi_b(v) &=\frac{d}{dt}_{|_{t=0}}\frac{d}{ds}_{|_{s=0}}(\theta_t\circ \gamma_s)\\
 &=\frac{d}{ds}_{|_{s=0}} \frac{d}{dt}_{|_{t=0}}(\theta_t\circ \gamma_s) \\
 &= \frac{d}{ds}_{|_{s=0}} X_b (\gamma_s)\\
 &= dX_b (v)
 \end{split}
\end{equation} seen as a section $X_b:M\ra TM$. Recall that $g$ induces a Riemannian metric on $TM$, denoted $g$ as well, so the $g$--orthogonal decomposition into the vertical and horizontal parts of $T_{(z,v)}(TM)=T_v(T_zM) \oplus \mH^{\nabla}$ gives $dX_b (v)=\nabla_vX_b+\tilde{v}$. Hence, identifying $T_v(T_zM)=T_zM$, we obtain
\begin{equation*}
  \Delta^{g}\mu_b =  2\sum_{i=1}^{n}\omega(\nabla_{w_i} X_b, w_i) =2 \sum_{i=1}^{n}\omega(d\phi_b(w_i), w_i)= -2 \sum_{i=1}^{n-m}\langle d\phi_i,b\rangle.
\end{equation*} \end{proof}

A direct consequence of this Lemma, applied on K\"ahler orbifold $(W,\sigma_o,\check{J}_o,\check{g}_o)$, is that for $\xi=X_b$, the function
$$-2f_{o,\xi}(z) = -(\Delta^{g_o}\langle\check{\mu}_o, [b]\rangle)_z = 2\langle \sum_{j=1}^{n-m_Z} \underline{\kappa}^Z_j, [b]\rangle =2\langle \sum_{j=1}^{n-m_Z} \underline{\kappa}^Z_j, b\rangle$$ when restricted on $Z$, is a rational polynomial. Actually it is a linear function with rational coefficients. Recall from~\S\ref{secWeights} that $\underline{\kappa}^Z_1,\dots,\underline{\kappa}^Z_{n-m_Z}$ lie in the annihilator of $b_o\in \kt$ and thus are defined on $\kt/\bR b_o$.    

This completes the proof of Theorem \ref{theoSCAL=rat}.

 \begin{corollary} \label{coroSboundary2} When $b\in \kt^+$ tends to the boundary of $\kt^+$, $\bfS_{X_b}$ tends to $+\infty$.  
 \end{corollary}
 \begin{proof} This goes essentially as in proof of Corollary~\ref{coroVboundary}. When $b\in \kt^+$ tends to the boundary of $\kt^+$ (but away from $0\in \del \kt^+$), there is at least one point $p\in N$ such that $\eta_o(X_b)_p \ra 0$ by definition of $\kt^+$, see \eqref{eq:SasakiCone}. Therefore, $p$ is a minimum of the function $\eta_o(X_b) =\langle\mu_{\eta_o},b \rangle$ and arguing as in the proof of Corollary~\ref{coroVboundary} we get that there is a unique connected component $\tilde{Z}\subset \mbox{crit}(\eta_o(X_b)) \subset N$ such that $\tilde{Z}$ is a local minimum, thus it is a global minimum and there is no other part of $N$ on which $\eta_o(X_b)$ tends to $0$ faster than it does on $\tilde{Z}$. Moreover, 
$$\langle \underline{\kappa}_i^Z , b\rangle >0$$ 
for $i=1,\dots , n-m_Z$ and, see~\S\ref{secWeights}, we have $\langle \underline{\kappa}_0^Z , b\rangle=\eta_o(X_b)_p$.  Hence, the leading term of $\bfS_{X_b}$  in \eqref{eqSCALsumZ0} is the one where $\langle \underline{\kappa}_0^Z , b\rangle$ appears with the greater exponent. This term is
 \begin{equation}\label{leadSxi}
 \frac{-2f_{o,\xi}(Z)}{d_Z(n-1)!}  \left(\prod_{j=0}^{n-m_Z} \frac{1}{\langle \underline{\kappa}_j^Z, b\rangle} \right)  \int_Z \left(\frac{c_1(E^Z_0)}{\langle \underline{\kappa}_0^Z, b\rangle}\right)^{m_Z}.
\end{equation} It is positive by the observations above, Lemma~\ref{lemFIXlap} and since $c_1(E_0^Z) = \iota_Z^*c_1(\mL)= \iota_Z^*[\sigma/2\pi] >0$ (see \S\ref{secWeights}) and tends to $+\infty$ when $\langle \underline{\kappa}_0^Z , b\rangle=\eta_o(X_b)_p \ra 0$.\end{proof}

\section{The Einstein-Hilbert functional and Scalar Curvature}

Here by scalar curvature we mean the transverse scalar curvature $s^T_\xi$ of a Sasakian structure $\cals=(\Ds,J, \xi,\eta,g)$.  Of course, $s^T_\xi$ is related to the scalar curvature $s_g$ of the Sasaki metric $g$ by $s_g=s^T_\xi-2n$ \cite{BG:book}.
Recall the Einstein-Hilbert functional \eqref{HEdef} which can be written in terms of the average scalar curvature $\bar{s}^T_\xi$ as
$$\bfH(\xi)=\frac{\bfS_\xi^{n+1}}{\bfV_\xi^n}=(\bar{s}^T_\xi)^{n+1}\bfV_\xi.$$
Since it is invariant under the transverse homothety operation $\xi\mapsto a^{-1}\xi$ with $a\in\bbr^+$, it descents to a function on the space of rays  of the Sasaki cone $\gt^+$. We have

\begin{lemma}\label{bfHinfty}
When $b\in \kt^+$ tends to the boundary of $\kt^+$, $\bfH(X_b)$ tends to $+\infty$.  
\end{lemma}

\begin{proof} Both $\bfS$ and $\bfV$ tend to $+\infty$ when $b$ tends to the boundary of $\kt^+$. Moreover, considering the leading terms \eqref{leadVol} and \eqref{leadSxi} into Equation \eqref{HEdef} and putting $\langle \underline{\kappa}_0^Z , b\rangle=\eta_o(X_b)_p$ as a factor we see that $\bfH(X_b)$ behaves as 
$$ \frac{C(b)}{\langle \underline{\kappa}_0^Z, b\rangle^{m_Z+1}}$$ where $C(b)$ is a rational function tending to $$f_{o,\xi}(Z)^{n+1}\left(\prod_{j=1}^{n-m_Z} \frac{1}{\langle \underline{\kappa}_j^Z, b\rangle} \right) \left( \int_Z c_1(E^Z_0)^{m_Z}\right)>0$$
%
when $\langle \underline{\kappa}_0^Z , b\rangle=\eta_o(X_b)_p \ra 0$. 
\end{proof}

Equation \eqref{eqSCALT} has several interesting consequences. Recall the definition of the {\it type} of a Sasakian structure. A Sasakian structure $(N,\xi,\eta,\Phi,g)$ is said to be of {\it positive (negative)} if $c_1(\calf_\xi)$ can be represented by a positive (negative) definite basic $(1,1)$ form. If $c_1(\calf_\xi)$ vanishes it of {\it null type}. When none of these hold it is of {\it indefinite type}. If $\dim~\gt^+\geq 2$, then the Sasakian structure is either of positive or indefinite type.

\begin{lemma}\label{bfSlem} 
Fix a Sasaki structure $\cals_0=(\Ds, J, \xi_0)$ such that its Sasaki automorophism group has dimension at least two.
\begin{enumerate}
\item There is no more than one ray of Reeb vector fields $\xi\in\kt^+$ having vanishing transverse scalar curvature $s^T_0$ and if there is one such ray $\gr_0$ then for all $\xi\in\gt^+\setminus \{\gr_0\}$ we have $\bfS_\xi>0$. 
\item If $\cals_0$ satisfies $s_0^T>0$ almost everywhere, then $\bfS_\xi>0$ for all $\xi\in\gt^+$. In particular, if $\cals_0$ is of positive type, then $\bfS_\xi>0$ for all $\xi\in\gt^+$.
\end{enumerate}
\end{lemma}

Note that if (1) of Lemma \ref{bfSlem} holds, then the ray $\gr_0$ is the unique global minimum of the Einstein-Hilbert functional $\bfH$ and $\bfH(\xi)=0$ for all $\xi \in \gr_0$. Moreover, no $\xi\in\gt^+$ has strictly positive scalar curvature. Somewhat more generally we have

\begin{proposition}\label{s0indef}
If the dimension of $\gt^+$ is at least two and there exists $\xi\in \gt^+$ such that $\bfS_{\xi}\leq 0$, then all elements of the Sasaki cone are of indefinite type.
\end{proposition}

\begin{proof}
Since ${\rm Dim}~\gt^+>1$, a Sasakian structure is either positive or indefinite. Suppose to the contrary that there is a Sasakian structure in $\gt^+$ of positive type, then taking this to be $\xi_o$ in \eqref{eqSCALT} gives a contradiction.
\end{proof}

\subsection{A Global Minimum}
We now show that $\bfH$ always attains a global minimum.
\begin{theorem}\label{globalmin}
Let $(N,\cald)$ be a contact manifold of Sasaki type with a $T$ action of Reeb type. Then there exists $\xi_{min}\in\gt^+$ that minimizes $\bfH$. Moreover, if $\bfH(\xi_{min})\neq 0$ and the corresponding Sasaki metric is extremal, it must have constant scalar curvature.
\end{theorem}

\begin{proof}
Recall \cite{BHLT_EH} that a set $\Sigma$ is a transversal subset of $\kt^+$ if $\Sigma\subset\kt^+_k$ and that $\Sigma$ meets each ray passing through $\kt^+$ in a single point and is a codimension one relatively compact subset of $\kt^+$  whose closure does not contain $0$. Let $\bar{\gt}^+$ denote the closure of $\gt^+$. Since $\bfH(\xi)$ is scale invariant, $H|_\grS$ is independent of $\grS$. We choose $\grS$ to be any intersection of $\bar{\gt}^+$ with a transverse hyperplane. Then $P=\bar{\gt}^+\cap \grS$ is a simple compact convex polytope. Now by Lemma \ref{bfHinfty} $\bfH(\xi)$ tends to $+\infty$ as $\xi$ tends to the boundary $\partial P$. Thus, since $\bfH$ is a continuous function on the interior $P^o=P\setminus\del P$ there exists $\xi_{min}$ such that $\bfH(\xi)\geq \bfH(\xi_{min})$. This proves the first statement.

For the second statement we assume that $\bfS_{\xi_{min}}\neq 0$. Then it follows from \cite{BHLT_EH} that the Sasaki-Futaki invariant $\bfF_\xi$ vanishes at $\xi=\xi_{min}$. So if $\cals_{min}$ is extremal it has constant scalar curvature. 

\end{proof}


Note that the ray $\gr_{min}$ of $\xi_{min}$ is generally not necessarily unique. However, by (1) of Lemma \ref{bfSlem}, it is unique if there exists a Sasakian structure in $\gt^+$ with vanishing transverse scalar curvature in which case $\bfH(\xi)\geq 0$ for all $\xi\in\gt^+$ and  equality holds if and only if $\xi=a\xi_{min}$ where $a\in\bbr^+$. In this case we get a unique global minimum. Explicit examples are given by Proposition 5.17 of \cite{BoTo13}.


There are counterexamples to the second statement of Theorem \ref{globalmin} if we drop the hypothesis that $\cals_{min}$ is extremal. See \cite{BovC16} for many examples where the entire Sasaki cone is relatively K-unstable and admits no extremal Sasaki metrics.

Another functional of interest to handle the case when $\bfS_\xi<0$ and $n$ is odd, is defined by
\begin{equation}\label{H1}
\bfH_1(\xi)={\rm sign}(\bfS_\xi)|\bfH(\xi)|.
\end{equation}
It follows as in the proof of Theorem \ref{globalmin} that $H_1$ also has a global minimum $\xi_{1min}$ giving
$$\bfH_1(\xi)\geq \bfH_1(\xi_{1min})$$
for all $\xi\in\gt^+$.

\subsection{Positivity and Rational Curves}
Recall that a projective algebraic variety $X$ of complex dimension $n$ is said to be {\it uniruled} if there exists a dominant rational map $X'\times\bbc\bbp^1\dashrightarrow X$ where $X'$ is a variety of dimension $n-1$. It is well known that a uniruled variety has Kodaira dimension $-\infty$, and that there is a rational curve through each point. Its relation with Sasakian geometry was mentioned briefly in Theorem 7.5.33 of \cite{BG:book} which states that a compact quasi-regular Sasakian structure of positive type has a uniruled quotient variety. This follows directly by a result of Miyaoka and Mori \cite{MiMo86}. In the simply connected case there is a topological classification of positive Sasakian 5-manifolds due to Kollar \cite{Kol05b} which shows how the existence of irrational curves determines the torsion in $H_2(N^5,\bbz)$ (see also Section 10.2.1 of \cite{BG:book}). However, a recent result of Heier and Wong \cite{HeWo12} shows that the hypothesis that the Sasaki manifold be of positive type can 
be weakened considerably. 

\begin{theorem}[Heier-Wong]\label{HWthm}
Let $N$ be a regular compact Sasaki manifold with Reeb field $\xi$. 
\begin{enumerate}
\item If $\bfS_\xi>0$, then its quotient variety is uniruled. 
\item If $\bfS_\xi=0$, then either its quotient is uniruled or its canonical line bundle is torsion. In particular, if the dimension of $\gt^+>1$, the quotient is uniruled.
\end{enumerate}
\end{theorem}

\begin{rem}
It appears quite likely that this result can be generalized to the quasi-regular case, although one would need to work with the orbifold canonical bundle. This seemingly entails generalizing the methods of Boucksom-Demailly-P\u{a}un-Peternell \cite{BDPP13} to the case when the variety $X$ is a normal projective variety with cyclic quotient singularities. Nevertheless, below we are able to generalize Theorem \ref{HWthm} to a certain class of orbifolds.
\end{rem}

The non-negativity of a particular $\bfS_\xi$ is, however, far from the last word. Indeed, Example 5.16 of \cite{BoTo13} gives a quasi-regular Sasakian structure on an $S^3$ bundle over a Riemann surface of genus 23 with constant transverse scalar curvature equal to $-16\pi$. This is not only uniruled, but a ruled manifold. 

Here we are content to consider a special case.
\begin{theorem}\label{orbiHW}
Let $N$ be a compact manifold of dimension $2n + 1$ with a CR-structure $(\Ds, J)$ of Sasaki type with Sasaki cone $\gt^+$ of dimension at least two. Suppose further that all quasi-regular $\xi\in\gt^+$ have $S^1$ quotients of the form $(W_\xi,\grD_\xi)$ where $W_\xi$ is a smooth projective variety and $\grD_\xi$ is a branch divisor depending on $\xi$. If $\bfS_\xi\geq 0$ then $W_\xi$ is uniruled. 
\end{theorem}

\begin{proof}
Recall \cite{BG:book} that a branch divisor has the form
$$\grD=\sum_i\bigl(1-\frac{1}{m_i}\bigr)D_i$$
where $D_i$ is an irreducible hypersurface contained in the orbifold singular locus and $m_i$ is the ramification index of $D_i$ and the sum is finite. The orbifold first Chern class and the first Chern class of $X$ are related by
\begin{equation}\label{orbiChern}
c_1^{orb}(W,\grD)=c_1(W)-\sum_i\bigl(1-\frac{1}{m_i}\bigr)c_1(L_i)
\end{equation}
where $L_i$ is the line bundle associated to $D_i$.
Now $2\pi c_1(W)$ is represented by the Ricci form $\grr_W$, and $2\pi c_1^{orb}(W,\grD)$ is represented by the orbifold Ricci form $\grr^{orb}$ which is the pushforward of the transverse Ricci form $\grr^T$ which is well defined since it is basic.
Let $\gro$ be any K\"ahler form on $X$, then by the work of Demailly \cite{Dem92} $2\pi c_1(L_i)$ is represented by the curvature current $c(L_i)$ of a singular Hermitian metric on $L_i$. Now Equation \eqref{orbiChern} implies
\begin{equation}\label{Riceqn}
\int_W\grr\wedge \gro^{n-1}=\int_W\grr^{orb}\wedge \gro^{n-1}+\sum_i\bigl(1-\frac{1}{m_i}\bigr) \int_W c(L_i)\wedge \gro^{n-1}.
\end{equation}
Note that the scalar curvature of the orbifold K\"ahler metric $\gro_0$ is just the transverse scalar curvature of the corresponding Sasakian structure.  The cohomology class $[\gro_0]\in H^2(W,\bbq)$ defines a positive orbi-line bundle $\call$ on $(W,\grD)$, so it is orbi-ample. It follows from orbifold theory that $\call^l$ where $l={\rm lcm}(m_i)$ is an ample line bundle on $X$. Thus $l[\gro_0]\in H^2(W,\bbz)$ and since $\call^l$ is ample $l[\gro_0]$ can be represented by a K\"ahler form which we choose to be $\gro$. This implies 
$$\int_W\grr^{orb}\wedge \gro^{n-1}=l^{n-1}\int_W\grr^{orb}\wedge \gro^{n-1}_0=\frac{l^{n-1}}{2n}\int_Ws^T\gro_0^n=C\bfS_\xi$$
for some $C>0$. If we let $\bfS_W$ denote the total scalar curvature of $X$, this and Equation \eqref{Riceqn} implies
$$\bfS_W=C\bfS_\xi +\sum_i\bigl(1-\frac{1}{m_i}\bigr) \int_W c(L_i)\wedge \gro^{n-1}.$$
Since the $D_i$ are effective it follows from (b) of Proposition 4.2 of \cite{Dem92} that $c(L_i)$ satisfies $c(L_i)\geq \gre\gro$ for some $\gre>0$, so this equation implies
\begin{equation}\label{orbiineq}
\bfS_W\geq C\bfS_\xi +\gre n!\sum_i\bigl(1-\frac{1}{m_i}\bigr) \bfV_W.
\end{equation}
So $\bfS_\xi\geq 0$ implies $\bfS_W>0$ and the result follows by the Heier-Wong Theorem \ref{HWthm}. 
\end{proof}

Actually by looking at sequences of quasi-regular Reeb fields that approach the boundary of $\gt^+$, we can do better.
\begin{corollary}\label{Reebseqcor}
Assuming the hypothesis of Theorem \ref{orbiHW},  if there exists a sequence $\{\xi_k\}$ of quasi-regular Reeb fields tending to the boundary in $\gt^+$ and such that $W_\xi=W_{\xi_k}$ is independent of $k$, then $W_\xi$ is uniruled.
\end{corollary}

\begin{proof}
Since the dimension of $\gt^+$ is a least $2$ we know by Corollary \ref{coroSboundary2} that $\bfS_\xi\rightarrow +\infty$ near the boundary. So there exists a quasi-regular $\xi\in\{\xi_k\}\subset \gt^+$ such that $\bfS_\xi>0$.
\end{proof}

Our assumption that all quasi-regular $\xi\in\gt^+$ have a codimension one orbifold singular set is restrictive. It is, however, realized by a fairly large collection of Sasaki families, for example, those coming from the $S^3_\bfw$-join studied in \cite{BoToJGA}. For these one can also easily construct sequences satisfying the conditions of Corollary \ref{Reebseqcor}.

\bibliographystyle{abbrv}

\end{document}